\documentclass[a4paper,11pt]{amsart}

\usepackage{amsfonts,latexsym,rawfonts,amsmath,amssymb,amsthm, mathrsfs, lscape}
\usepackage[all]{xy}
\usepackage[english]{babel}
\usepackage[utf8]{inputenc}
\usepackage[T1]{fontenc}
\usepackage{graphicx}
\usepackage{booktabs}
\usepackage{array, tabularx}
\usepackage{setspace}
\usepackage{textcomp}
\usepackage{tikz}
\usepackage{multirow}
\usepackage{paralist}
\usepackage{comment}

\usepackage[hypertexnames=false,
backref=page,
    pdftex,
    pdfpagemode=UseNone,
    breaklinks=true,
    extension=pdf,
    colorlinks=true,
    linkcolor=blue,
    citecolor=blue,
    urlcolor=blue,
]{hyperref}

\usepackage[top=1.5in, bottom=.9in, left=0.9in, right=0.9in]{geometry}

\newcommand{\referenza}{}

\newtheorem{thm}{Theorem}[section]
\newtheorem*{thm*}{Theorem \referenza}
\newtheorem{cor}[thm]{Corollary}
\newtheorem*{cor*}{Corollary \referenza}
\newtheorem{lem}[thm]{Lemma}
\newtheorem*{lem*}{Lemma \referenza}

\newtheorem{prop}[thm]{Proposition}
\newtheorem*{prop*}{Proposition \referenza}

\newtheorem*{conj*}{Conjecture \referenza}
\newtheorem{rmk}[thm]{Remark}

\newtheorem{defi}[thm]{Definition}

\numberwithin{equation}{section}

\def \N {\mathbb N}

\def \Q {\mathbb Q}
\def \R {\mathbb R}
\def \C {\mathbb C}
\def \Z {\mathbb Z}

\newcommand{\sspace}{\text{\--}}

\DeclareMathOperator{\codim}{codim}

\allowdisplaybreaks[1]

\title[Hermitian ranks]{Hermitian ranks of compact complex manifolds}

\author{Daniele Angella}
\address[Daniele Angella]{
Dipartimento di Matematica e Informatica ``Ulisse Dini''\\
Universit\`a degli Studi di Firenze\\
viale Morgagni 67/a\\
50134 Firenze, Italy
}
\email{daniele.angella@gmail.com}
\email{daniele.angella@unifi.it}

\author{Adriano Tomassini}
\address[Adriano Tomassini]{Dipartimento di Scienze Matematiche, Fisiche ed Informatiche\\
Plesso Matematico e Informatico\\
Università di Parma\\
Parco Area delle Scienze 53/A\\
43124 Parma, Italy
}
\email{adriano.tomassini@unipr.it}

\thanks{The first author is supported by the Project PRIN ``Varietà reali e complesse: geometria, topologia e analisi armonica'', by the Project FIRB ``Geometria Differenziale e Teoria Geometrica delle Funzioni'', by SIR2014 project RBSI14DYEB ``Analytic aspects in complex and hypercomplex geometry'', and by GNSAGA of INdAM.
The second author is supported by Project PRIN ``Varietà reali e complesse: geometria, topologia e analisi armonica'' and by GNSAGA of INdAM}

\date{\today}

\begin{document}

\begin{abstract}
 We investigate degenerate special-Hermitian metrics on compact complex manifolds, in particular, degenerate K\"ahler and locally conformally K\"ahler metrics on special classes of non-K\"ahler manifolds.
\end{abstract}

\maketitle

\section*{Introduction}
In this note, we continue in studying special-Hermitian structures on compact complex manifolds, in view of a (far-to-be-obtained) possible classification of compact complex non-K\"ahler manifolds.

In particular, we focus on the existence of degenerate special-Hermitian metrics. We investigate here degenerate K\"ahler metrics and degenerate locally conformally K\"ahler metrics, by introducing the notion of special-Hermitian ranks, as a first development of the K\"ahler rank by R. Harvey and H.~B. Lawson and by I. Chiose in the non-K\"ahler setting.
We investigate here classes of non-K\"ahler examples.

\medskip

R. Harvey and H. B. Lawson provided an intrinsic characterization of the K\"ahler condition. In \cite[Theorem (14)]{harvey-lawson}, they proved that, on a compact complex manifold $X$, one and only one of the following facts holds:
\begin{inparaenum}[\itshape (i)]
 \item there is a positive $(1,1)$-form being closed; (namely, a K\"ahler metric;)
 \item there is a (non-trivial) positive bidimension-$(1, 1)$-current being component of a boundary.
\end{inparaenum}
 
This let them to introduce a notion of K\"ahler rank for compact complex surfaces, in terms of the foliated set
$$ \mathcal B(X) \;:=\; \left\{ x \in X \;:\; \exists \varphi \in P^\infty_{\text{bdy}}(X) \text{ such that }\varphi_x\neq 0 \right\} \;, $$
where $P^\infty_{\text{bdy}}(X)$ denotes the subcone of smooth currents in the cone of positive bidimension-$(1,1)$-currents on $X$ being a boundary.

By \cite[Corollary 4.3]{chiose-toma}, the K\"ahler rank of a compact complex surface $X$ is equal to the maximal rank that a non-negative closed $(1, 1)$-form may attain at some point of $X$. (See also \cite[Definition 1.2]{fino-grantcharov-verbitsky}.) This allows I. Chiose to extend the notion of K\"ahler rank to higher-dimensional compact complex manifolds, \cite[Definition 1.1]{chiose}, by setting
\begin{equation}\label{eq:def-Kr}
\mathrm{Kr}(X) \;:=\; \max \left\{ k\in\N \;:\; \exists \omega\in \wedge^{1,1}X \text{ s.t. } \omega\geq0,\; d\omega=0, \text{ and }\omega^k\neq 0 \right\} \;.
\end{equation}

\medskip

By relaxing the K\"ahler condition, several notions of special-Hermitian metrics can be defined: {\itshape e.g.} Hermitian-symplectic, balanced in the sense of Michelsohn \cite{michelsohn}, pluri-closed \cite{bismut}, astheno-K\"ahler \cite{jost-yau}, Gauduchon \cite{gauduchon}, strongly-Gauduchon \cite{popovici}, and others.
The notion in Equation \eqref{eq:def-Kr} can be restated for some of these metrics: we will consider {\itshape e.g.} the SKT case \eqref{eq:SKT-rank}.

In particular, we introduce and study the {\em Hermitian locally conformally K\"ahler rank}.
It is defined in Equation \eqref{eq:HlcK-rank}. Essentially, we replace the condition of $d\omega=0$ in Equation \eqref{eq:def-Kr} by $d\omega-\vartheta\wedge\omega=0$ for some $d$-closed $1$-form $\vartheta=0$. By the Poincar\'e Lemma, $\vartheta$ is locally $d$-exact: $\vartheta\stackrel{\text{loc}}{=}dg$. Then $\exp(-g)\omega$ is a local conformal change of $\omega$ being K\"ahler.

Both the K\"ahler and the lcK conditions are cohomological in nature. Moreover, cohomologies of nilmanifolds (namely, compact quotients of connected simply-connected nilpotent Lie groups,) can be often reduced as Lie algebra invariants. It follows that the K\"ahler rank and the lcK rank of nilmanifolds is often encoded in the Lie algebra, see Lemma \ref{lem:inv-ranks}.
This allows us to study explicitly the K\"ahler rank and the lcK rank of $6$-dimensional nilmanifolds. (With the possible exception of nilmanifolds associated to the Lie algebra $\mathfrak h_7=(0,0,0,12,13,23)$ in the notation of Salamon \cite{salamon}, see \cite{rollenske-survey}.) This is done in Section \ref{sec:nilmanifolds-ranks}. Compare also \cite[Section 4.2]{fino-grantcharov-verbitsky}, where the same results have been obtained independently.

As a further example, we consider a non-K\"ahler manifold obtained as a torus-suspension in \cite{magnusson}.
The investigation of these manifolds was suggested by Valentino Tosatti and may deserve further studies in non-K\"ahler geometry.

\bigskip

\noindent{\sl Acknowledgments.}
The authors would like to thank Valentino Tosatti for useful discussions.
Many thanks also to the anonymous Referee for valuable comments.

\section{Hermitian ranks}
In this section, we recall the definitions of K\"ahler rank for compact complex surfaces by R. Harvey and H.~B. Lawson, and for compact complex manifolds by I. Chiose. On the same lines, we also introduce the notions of lcK rank and pluri-closed rank.

\subsection{K\"ahler rank}
Let $X$ be a compact complex surfaces.
Denote by $P_{\text{bdy}}(X)$ the cone of positive bidimension-$(1,1)$-currents on $X$ being a boundary, and denote by $P^\infty_{\text{bdy}}(X)$ the subcone of smooth currents.
On a compact complex surface $X$, the cone $P^\infty_{\text{bdy}}(X)$ coincides with the cone $P^\infty_{\text{bdy}_{1,1}}(X)$ of positive bidimension-$(1,1)$-currents on $X$ being component of a boundary, and any form $\varphi\in P^\infty_{\text{bdy}}(X)$ is simple, ({\itshape i.e.} of rank less or equal than one,) at every point of $X$, \cite[Proposition (37)]{harvey-lawson}.
Set:
$$ \mathcal B(X) \;:=\; \left\{ x \in X \;:\; \exists \varphi \in P^\infty_{\text{bdy}}(X) \text{ such that }\varphi_x\neq 0 \right\} \;. $$
The open subset $\mathcal B(X)\subseteq X$ carries an intrinsically defined complex analytic foliation $\mathcal F$, which is characterized by the property that $\varphi\lfloor_{\mathcal F}=0$ for any $\varphi\in P^\infty_{\text{bdy}}(X)$, \cite[Theorem 40]{harvey-lawson}.
The {\em K\"ahler rank} of the compact complex surface $X$ \cite[Definition 41]{harvey-lawson} is defined to be:
\begin{inparaenum}[\itshape (a)]
 \item two, when $X$ admits K\"ahler metrics; (that is, the open subset $\mathcal B(X)$ in $X$ is empty;)
 \item one, when the complement of the open subset $\mathcal B(X)$ in $X$ is contained in a complex curve and non-empty;
 \item zero, otherwise.
\end{inparaenum}

The K\"ahler rank of compact complex surfaces is a bimeromorphic invariant, \cite[Corollary 4.1]{chiose-toma}.
Surfaces with even first Betti number have K\"ahler rank two, \cite{lamari, buchdahl}.
Elliptic non-K\"ahler surfaces have K\"ahler rank one, \cite[page 187]{harvey-lawson}.
Non-elliptic non-K\"ahler surfaces are in class VII: under the GSS conjecture, their minimal model is one of the following:
\begin{inparaenum}[\itshape (i)]
  \item Inoue surfaces: K\"ahler rank is one, \cite[\S10]{harvey-lawson};
  \item Hopf surface: K\"ahler rank is one or zero according to the type, \cite[\S9]{harvey-lawson};
  \item Kato surfaces: K\"ahler rank is zero, \cite{chiose-toma}.
\end{inparaenum}

\medskip

Now, let $X$ be a compact complex manifold of complex dimension $n\geq2$. Notice that, when $n=2$, the K\"ahler rank, as in \cite[Definition 41]{harvey-lawson}, is equal to the maximal rank that a non-negative closed $(1,1)$-form may attain at some point of $X$, thanks to \cite[Corollary 4.3]{chiose-toma}. Then, the following definition by I. Chiose is coherent. Compare also \cite[Definition 1.2]{fino-grantcharov-verbitsky}.

\begin{defi}[{\cite[Definition 1.1]{chiose}}]
 Let $X$ be a compact complex manifold of complex dimension $n$. The {\em K\"ahler rank} of $X$ is defined to be
 $$ \mathrm{Kr}(X) \;:=\; \max \left\{ k\in\N \;:\; \exists \omega\in \wedge^{1,1}X \text{ s.t. } \omega\geq0,\; d\omega=0, \text{ and }\omega^k\neq 0 \right\} \;\in\; \left\{0,\ldots,n\right\} \;. $$
\end{defi}

\begin{rmk}
 Note that, by \cite[Theorem 0.2]{chiose}, for compact complex manifolds of complex dimension $3$ with maximal K\"ahler rank, the ``tamed-to-compatible'' conjecture, \cite[page 678]{li-zhang}, \cite[Question 1.7]{streets-tian}, holds.
\end{rmk}

\subsection{Locally conformally K\"ahler rank}
Now, we consider {\em locally conformally K\"ahler} structures \cite{dragomir-ornea}. Such a structure is given by $(\vartheta, \omega)$, where $\vartheta$ is a closed $1$-form and $\omega$ is a Hermitian metric satisfying $d_\vartheta\omega=0$, where
$$ d_\vartheta \;:=\; d-\vartheta\wedge \;. $$
Note that, locally, $\vartheta\stackrel{\text{loc}}{=}df$ for some smooth function $f$. Therefore $\exp(-f)\omega$ is a local K\"ahler structure being a locally conformally transformations of $\omega$.

We now admit degenerate metrics, and we introduce the following rank.
(We add the adjective ``Hermitian'' in order to avoid confusion with the notion of lcK rank introduced in \cite{gini-ornea-parton-piccinni}, which regards the Lee form $\vartheta$.)

\begin{defi}
 Let $X$ be a compact complex manifold of complex dimension $n$. The {\em Hermitian locally conformally K\"ahler rank} of $X$ is defined to be
 \begin{eqnarray}\label{eq:HlcK-rank}
  \mathrm{HlcKr}(X) &:=& \max \left\{ k\in\N \;:\; \exists \vartheta \in \wedge^1X \text{ s.t. } d\vartheta=0,\right.\\[5pt]
  &&\left.\exists\omega\in \wedge^{1,1}X \text{ s.t. } \omega\geq0,\; d_{\vartheta}\omega=0, \text{ and }\omega^k\neq 0 \right\} \nonumber\\[5pt]
  &\in& \left\{0,\ldots,n\right\} \;.\nonumber
 \end{eqnarray}
\end{defi}

Clearly, since $d_{0}=d$, it holds $\mathrm{Kr}(X)\leq\mathrm{HlcKr}(X)\leq \dim X$. Moreover, if $X$ admits a locally conformally K\"ahler metric, then clearly $\mathrm{HlcKr}(X)=\dim X$.

\subsection{Pluri-closed rank}
Finally, we consider {\em pluri-closed metrics} \cite{bismut}, namely, Hermitian metrics $\omega$ such that $\partial\overline\partial\omega=0$, also called {\em SKT metrics}. We define the following.

\begin{defi}\label{def:skt-rank}
 Let $X$ be a compact complex manifold of complex dimension $n$. The {\em pluri-closed rank} of $X$ is defined to be
 \begin{eqnarray}\label{eq:SKT-rank}
 \mathrm{SKTr}(X) &:=& \max \left\{ k\in\N \;:\; \exists \omega\in \wedge^{1,1}X \text{ s.t. } \omega\geq0,\; \partial\overline\partial\omega=0, \text{ and }\omega^k\neq 0 \right\} \\[5pt]
 &\in& \left\{0,\ldots,n\right\} \;.\nonumber
 \end{eqnarray}
\end{defi}

Note that, $\mathrm{Kr}(X)\leq\mathrm{SKT}(X)$. Moreover, by \cite[Théorème 1]{gauduchon}, any compact Hermitian manifold admit a unique {\em Gauduchon metric} in the conformal class up to scaling, that is, a metric $\omega$ satisfying $\partial\overline\partial\omega^{n-1}=0$, where $n=\dim X$. In particular, it follows that, on compact complex surfaces, the pluri-closed rank is always maximum, equal to $2$.

\section{Special-Hermitian ranks of homogeneous manifolds of solvable Lie groups}
In this section, we investigate the K\"ahler and the Hermitian lcK ranks of homogeneous manifolds of solvable Lie groups.

\medskip

Let $X = \left.\Gamma \middle\backslash G \right.$ be a solvmanifold, namely, a compact quotient of a connected simply-connected solvable Lie group $G$ by a co-compact discrete subgroup $\Gamma$. Assume that $X$ is endowed with an invariant complex structure, (that is, the complex structure is induced by a complex structure on $G$ being invariant with respect to the action of $G$ on itself given by left-translations.)
Then $\wedge^{\bullet,\bullet}\mathfrak{g}^\ast \hookrightarrow \wedge^{\bullet,\bullet}X$ is a sub-complex.

In the definition of special-Hermitian ranks, we can restrict to invariant metrics: set the {\em invariant K\"ahler rank} and the {\em invariant Hermitian lcK rank} to be
\begin{eqnarray*}
 \mathrm{Kr}(\mathfrak g) &:=& \max \left\{ k\in\N \;:\; \exists \omega\in \wedge^{1,1}\mathfrak{g}^\ast \text{ s.t. } \omega\geq0,\; d\omega=0, \text{ and }\omega^k\neq 0 \right\} \;, \\[5pt]
 \mathrm{HlcKr}(\mathfrak g) &:=& \max \left\{ k\in\N \;:\; \exists \vartheta \in \wedge^1\mathfrak{g}^\ast \text{ s.t. } d\vartheta=0,\right.\\[5pt]
  &&\left.\exists\omega\in \wedge^{1,1}\mathfrak{g}^\ast \text{ s.t. } \omega\geq0,\; d_{\vartheta}\omega=0, \text{ and }\omega^k\neq 0 \right\} \;.
\end{eqnarray*}
They have the advantage to be easier to be computed. In general, it holds
$$ \mathrm{Kr}(\mathfrak g) \;\leq\; \mathrm{Kr}(X) \qquad \text{ and } \qquad \mathrm{HlcKr}(\mathfrak g) \;\leq\; \mathrm{HlcKr}(X) \;. $$
In fact, we prove that equalities hold, under the assumption that the map $H^{\bullet,\bullet}_{\overline\partial}(\mathfrak g)\to H^{\bullet,\bullet}_{\overline\partial}(X)$ induced by the inclusion $\wedge^{\bullet,\bullet}\mathfrak{g}^\ast\to\wedge^{\bullet,\bullet}X$ is an isomorphism. The assumption holds true, {\itshape e.g.} when $G$ is nilpotent and the complex structure is either holomorphically-parallelizable, or Abelian, or nilpotent, or rational, \cite{console-survey, rollenske-survey}. It holds true also when $X$ is a compact complex surface being diffeomorphic to a solvmanifold, \cite{angella-dloussky-tomassini}.
 
\begin{lem}\label{lem:inv-ranks}
 Let $X=\Gamma\backslash G$ be a solvmanifold. Assume that the map $H^{\bullet,\bullet}_{\overline\partial}(\mathfrak g)\to H^{\bullet,\bullet}_{\overline\partial}(X)$ induced by the inclusion $\wedge^{\bullet,\bullet}\mathfrak{g}^\ast\to\wedge^{\bullet,\bullet}X$ is an isomorphism.
 Then the Hermitian locally conformally K\"ahler rank $\mathrm{HlcKr}(X)$ and the invariant Hermitian locally conformally K\"ahler rank $\mathrm{HlcKr}(\mathfrak g)$ are equal. In particular, the K\"ahler rank $\mathrm{Kr}(X)$ and the invariant K\"ahler rank $\mathrm{Kr}(\mathfrak{g})$ are equal.
\end{lem}

\begin{proof}
 Clearly, $\mathrm{HlcKr}(\mathfrak g) \leq \mathrm{HlcKr}(X)$.
 Let $\vartheta$ be a $d$-closed $1$-form, and $\omega\geq0$ be a $(1,1)$-form satisfying $d_\vartheta\omega=0$ such that $\omega^k\neq0$.
 We show that there exist an invariant $d$-closed $1$-form $\hat\vartheta$ and an invariant $(1,1)$-form $\hat\omega\geq 0$ satisfying $d_{\hat\vartheta}\hat\omega$ such that $\hat\omega^k\neq0$.
 
 By the assumption and by the Fr\"olicher spectral sequence, the average map
 $$ \mu\colon \wedge^{\bullet,\bullet}X \to \wedge^{\bullet,\bullet}\mathfrak{g}^\ast \;,\qquad \mu(\alpha)\;:=\;\int_X \alpha\lfloor_m\,\eta(m) $$
 (here, $\eta$ is a bi-invariant volume forms, thanks to Milnor,) induces the identity in de Rham cohomology.
 In particular, $\hat\vartheta:=\mu(\vartheta)$ is an invariant $d$-closed $1$-form, and there exists $f$ smooth function such that $\hat\vartheta=\vartheta+df$.
 
 Note that $d_{\vartheta+df} = \exp(f) \cdot d_\vartheta (\exp(-f)\cdot \sspace)$.
 Therefore $\tilde\omega:=\exp(f)\cdot \omega\geq 0$ is a $(1,1)$-form satisfying $d_{\hat\vartheta}\tilde\omega=0$ such that $\tilde\omega^k=\exp(kf)\omega^k\neq0$. In particular, $[\tilde\omega]\in H^{2}_{d_{\hat\vartheta}}(X)$.
 By \cite{hattori}, the average map $\mu$ induces the identity also in the cohomology of the twisted differential $d_{\hat\vartheta}$.
 So we get that $\hat\omega:=\mu(\tilde\omega)\geq0$ is an invariant $(1,1)$-form satisfying $d_{\hat\vartheta}\hat\omega=0$, and there exists $\alpha$ a $1$-form such that $\hat\omega=\tilde\omega+d_{\hat\vartheta}\alpha$.
 
 Moreover, we have
 $$ \hat\omega^k \;=\; \omega^k + d_{k\hat\vartheta} \varphi \qquad \text{ where } \varphi \;:=\; \sum_{\substack{s+t=k\\t\geq1}} {k \choose s} \cdot \tilde\omega^s\wedge \alpha\wedge (d_{(t-1)\hat\vartheta}\alpha)^{t-1} \;. $$
 That is, $[\hat\omega^k]=[\tilde\omega^k]$ in $H^{2k}_{d_{k\hat\vartheta}}(X)$. Therefore, since $\hat\omega^k$ is invariant, it follows that $\mu(\tilde\omega^k)=\hat\omega^k$.
 Since $\tilde\omega^k\geq0$ and $\tilde\omega^k\neq0$, then $\hat\omega^k\neq0$.
 
 As for the case of K\"ahler rank, it suffices to note that $[\vartheta]=0$ if and only if $[\hat\vartheta]=0$.
\end{proof}

As a direct consequence, we have the following.

\begin{cor}
 Let $X=\Gamma\backslash G$ be a solvmanifold. Assume that the map $H^{\bullet,\bullet}_{\overline\partial}(\mathfrak g)\to H^{\bullet,\bullet}_{\overline\partial}(X)$ induced by the inclusion $\wedge^{\bullet,\bullet}\mathfrak{g}^\ast\to\wedge^{\bullet,\bullet}X$ is an isomorphism.
 Then the K\"ahler rank $\mathrm{Kr}(X)$, (respectively, the Hermitian locally conformally K\"ahler rank $\mathrm{HlcKr}(X)$,) is maximum if and only if there exists a Hermitian metric being K\"ahler, (respectively, locally conformally K\"ahler.)
\end{cor}

\section{Special-Hermitian ranks of \texorpdfstring{$6$}{6}-dimensional nilmanifolds}\label{sec:nilmanifolds-ranks}
By using Lemma \ref{lem:inv-ranks}, we can compute the K\"ahler and Hermitian lcK ranks of $6$-dimensional nilmanifolds with invariant complex structures, except possibly for the nilmanifolds associated to the Lie algebra $\mathfrak h_7=(0,0,0,12,13,23)$ in the notation of Salamon \cite{salamon}.
In fact, the assumption of the map $H^{\bullet,\bullet}_{\overline\partial}(\mathfrak g)\to H^{\bullet,\bullet}_{\overline\partial}(X)$ induced by the inclusion being an isomorphism is satisfied, see \cite{console-survey, rollenske-survey}.

\medskip

It is well-known \cite{ugarte, ceballos-otal-ugarte-villacampa} that, up to (linear-)equivalence, the invariant complex structures on $6$-dimensional nilmanifolds are parametrized into the following families: there exists a global co-frame $\{\varphi^1,\varphi^2,\varphi^3\}$ of invariant $(1,0)$-forms such that the structure equations are
\begin{description}
 \item[(P)] $d\varphi^1=0$, $d\varphi^2=0$, $d\varphi^3=\rho\varphi^1\wedge\varphi^2$,\\
 where $\rho\in\{0,1\}$;
 \item[(I)] $d\varphi^1=0$, $d\varphi^2=0$, $d\varphi^3=\rho\varphi^1\wedge\varphi^2+\varphi^1\wedge\bar\varphi^1+\lambda\varphi^1\wedge\bar\varphi^2+D\varphi^2\wedge\bar\varphi^2$,\\
 where $\rho\in\{0,1\}$, $\lambda\in\R^{\geq0}$, $D\in\C$ with $\Im D\geq 0$;
 \item[(II)] $d\varphi^1=0$, $d\varphi^2=\varphi^1\wedge\bar\varphi^1$, $d\varphi^3=\rho\varphi^1\wedge\varphi^2+B\varphi^1\wedge\bar\varphi^2+c\varphi^2\wedge\bar\varphi^1$,\\
 where $\rho\in\{0,1\}$, $B\in\C$, $c\in\R^{\geq0}$, with $(\rho,B,c)\neq(0,0,0)$;
 \item[(III)] $d\varphi^1=0$, $d\varphi^2=\varphi^1\wedge\varphi^3+\varphi^1\wedge\bar\varphi^3$, $d\varphi^3=\varepsilon\varphi^1\wedge\bar\varphi^1\pm i\left(\varphi^1\wedge\bar\varphi^2-\varphi^2\wedge\bar\varphi^1\right)$,\\
 where $\varepsilon\in\{0,1\}$.
\end{description}

\subsection{K\"ahler rank of \texorpdfstring{$6$}{6}-dimensional nilmanifolds}
As for the K\"ahler rank, we have the following.

\begin{prop}\label{prop:kahler-rank-6-nilmfd}
 On $6$-dimensional nilmanifolds endowed with invariant complex structures, (except possibly for the nilmanifolds associated to the Lie algebra $\mathfrak h_7$,) the K\"ahler rank takes the following values:
 \begin{eqnarray*}
  \text{\normalfont\bfseries (P):}&\mathrm{Kr}(X)\;=\;& 3 \quad \text{ if } \rho=0 \;, \\[5pt]
  &\mathrm{Kr}(X)\;=\;& 2 \quad \text{ if } \rho=1 \;; \\[5pt]
  \text{\normalfont\bfseries (I):}&\mathrm{Kr}(X)\;=\;& 2 \;; \\[5pt]
  \text{\normalfont\bfseries (II):}&\mathrm{Kr}(X)\;=\;& 1 \quad \text{ if } (\rho,B,c)\neq(1,1,0)\;, \\[5pt]
  &\mathrm{Kr}(X)\;\geq\;& 1 \quad \text{ if } (\rho,B,c)=(1,1,0)\;; \\[5pt]
  \text{\normalfont\bfseries (III):}&\mathrm{Kr}(X)\;=\;& 1 \;;
\end{eqnarray*}
\end{prop}

\begin{proof}
Thanks to Lemma \ref{lem:inv-ranks}, we are reduced to compute the invariant ranks.
The arbitrary invariant $(1,1)$-form $\omega$ such that $\omega\geq0$ is
\begin{eqnarray}\label{eq:generic-metric}
 \omega &=& ir^2\varphi^{1}\wedge\bar\varphi^1 + is^2\varphi^{2}\wedge\bar\varphi^2 + it^2\varphi^{3}\wedge\bar\varphi^3 \\[5pt]
  \nonumber
  &&+\left(u\varphi^1\wedge\bar\varphi^2-\bar u\varphi^2\wedge\bar\varphi^1\right)
  +\left(v\varphi^2\wedge\bar\varphi^3-\bar v\varphi^3\wedge\bar\varphi^2\right)
  +\left(z\varphi^1\wedge\bar\varphi^3-\bar z\varphi^3\wedge\bar\varphi^1\right) \;,
\end{eqnarray}
where $r,s,t\in\R$, $u,v,z\in\C$ satisfy
\begin{eqnarray}
\label{eq:condition-generic-metric}
&&r^2 \;\geq\; 0\;,\qquad
s^2 \;\geq\; 0\;,\qquad
t^2 \;\geq\; 0\;,\\[5pt]
\nonumber
&&r^2s^2 \;\geq\; |u|^2\;,\qquad
s^2t^2 \;\geq\; |v|^2\;,\qquad
r^2t^2 \;\geq\; |z|^2\;,\\[5pt]
\nonumber
&&r^2s^2t^2+2\Re(i\bar u \bar v z) \;\geq\; t^2|u|^2+r^2|v|^2+s^2|z|^2 \;.
\end{eqnarray}

We have
\begin{eqnarray*}
\frac12\omega^2 &=&
 (r^2s^2-|u|^2) \, \varphi^{12\bar1\bar2}
+(-ir^2v-\bar uz) \, \varphi^{12\bar1\bar3}
+(is^2z-uv) \, \varphi^{12\bar2\bar3} \\[5pt]
&&+(ir^2\bar v-u\bar z) \, \varphi^{13\bar1\bar2}
+(r^2t^2-|z|^2) \, \varphi^{13\bar1\bar3}
+(-it^2u-\bar v z) \, \varphi^{13\bar2\bar3} \\[5pt]
&&+(-is^2\bar z-\bar u \bar v) \, \varphi^{23\bar1\bar2}
+(it^2\bar u-v\bar z) \, \varphi^{23\bar1\bar3}
+(s^2t^2-|v|^2) \, \varphi^{23\bar2\bar3}
\end{eqnarray*}
(for simplicity of notation, we shorten, {\itshape e.g.} $\varphi^{12\bar2\bar3}:=\varphi^{1}\wedge\varphi^{2}\wedge\bar\varphi^{2}\wedge\bar\varphi^{3}$,) and
\begin{eqnarray*}
\frac16\omega^3 &=&
 (ir^2s^2t^2-ir^2|v|^2-is^2|z|^2-it^2|u|^2+uv\bar z-\bar u \bar v z) \, \varphi^{123\bar1\bar2\bar3} \;.
\end{eqnarray*}

We compute:
\begin{eqnarray*}
 \text{\bfseries(P):}&
     \quad\partial\omega \;=\;&
     -\bar z\rho \, \varphi^{12\bar1}
     -\bar v\rho \, \varphi^{12\bar2}
     +i t^2 \rho \, \varphi^{12\bar3} \;; \\[5pt]
 \text{\bfseries(I):}&
     \quad\partial\omega \;=\;&
     (-v+\lambda z-\rho \bar z) \, \varphi^{12\bar1}
     + (-\bar v \rho + z\bar D) \, \varphi^{12\bar2}
     + (it^2\rho) \, \varphi^{12\bar3} \\[5pt]
     && + (-it^2) \, \varphi^{13\bar1}
     + (-it^2\lambda) \, \varphi^{23\bar1}
     + (-it^2\bar D) \, \varphi^{23\bar2}
     \;; \\[5pt]
 \text{\bfseries(II):}&
     \quad\partial\omega \;=\;&
     \left(-is^2+z\bar B-\bar z\rho\right) \, \varphi^{12\bar1}
     + \left(-cv-\bar v\rho\right) \, \varphi^{12\bar2}
     + \left(it^2\rho\right) \, \varphi^{12\bar3} \\[5pt]
     && + \left(\bar v\right) \, \varphi^{13\bar1}
     + \left(-it^2 c\right) \, \varphi^{13\bar2}
     + \left(-it^2\bar B\right) \, \varphi^{23\bar1}
     \;; \\[5pt]
 \text{\bfseries(III):}&
     \quad\partial\omega \;=\;&
     (\mp iz-\varepsilon v)\,\varphi^{12\bar1}
     + (\mp iv)\,\varphi^{12\bar2}
     + (u-\bar u-it^2\varepsilon)\,\varphi^{13\bar1} \\[5pt]
     && + (is^2\pm t^2)\,\varphi^{13\bar2}
     + (v)\,\varphi^{13\bar3}
     + (is^2\mp t^2)\,\varphi^{23\bar1}
     \;.
\end{eqnarray*}
The statement follows.
\end{proof}

\begin{rmk}
 The results in Proposition \ref{prop:kahler-rank-6-nilmfd} have been obtained independently in \cite[Section 4.2]{fino-grantcharov-verbitsky}.
\end{rmk}

\subsection{Hermitian locally conformally K\"ahler rank of \texorpdfstring{$6$}{6}-dimensional nilmanifolds}
As for the Hermitian lcK rank, we have the following.

\begin{prop}
 On $6$-dimensional nilmanifolds endowed with invariant complex structures, (except possibly for the nilmanifolds associated to the Lie algebra $\mathfrak h_7$,) the Hermitian locally conformally K\"ahler rank takes the following values:
 \begin{eqnarray*}
  \text{\normalfont\bfseries (P):}&\mathrm{HlcKr}(X)\;=\;& 3 \quad \text{ if } \rho=0 \;, \\[5pt]
  &\mathrm{HlcKr}(X)\;=\;& 2 \quad \text{ if } \rho=1 \;; \\[5pt]
  \text{\normalfont\bfseries (I):}&\mathrm{HlcKr}(X)\;=\;& 3 \quad\text{ if } (\rho,\lambda,D)=(0,0,-1) \;, \\[5pt]
  &\mathrm{HlcKr}(X)\;=\;& 2 \quad\text{ if } (\rho,\lambda,D)\neq(0,0,-1)\;; \\[5pt]
  \text{\normalfont\bfseries (II):}&\mathrm{HlcKr}(X)\;=\;& 2 \;; \\[5pt]
  \text{\normalfont\bfseries (III):}&\mathrm{HlcKr}(X)\;=\;& 1 \;;
\end{eqnarray*}
\end{prop}

\begin{proof}
By \cite[Main Theorem]{sawai}, a non-toral compact nilmanifold with a left-invariant complex structure has a locally conformally K\"ahler structure if and only if it is biholomorphic to a quotient of the Heisenberg group times $\R$. In particular, the only $6$-dimensional non-Abelian nilpotent Lie algebra admitting lcK structures is $\mathfrak{h}_3$, which appears in family {\bfseries (I)} with parameters $\rho=0$, $\lambda=0$, $D=-1$.

In case {\bfseries (II)}, consider the $d$-closed $1$-form $\vartheta:=\varphi^2+\bar\varphi^2$ and the $d_\vartheta$-closed $2$-form, $\Omega:=i\,\varphi^1\wedge\bar\varphi^1+i\,\varphi^2\wedge\bar\varphi^2\geq0$.

In case {\bfseries (III)}, the arbitrary $d$-closed $1$-form is $\vartheta=\vartheta_1\varphi^1+\vartheta_3\varphi^3+\bar\vartheta_1\bar\varphi^1+\vartheta_3\bar\varphi^3$, where $\vartheta_1\in\C$ and $\vartheta_3\in\R$.
By straightforward computations, which we performed with the aid of Sage \cite{sage}, we get that the arbitrary form $\omega$ in \eqref{eq:generic-metric} with conditions \eqref{eq:condition-generic-metric} is $d_\vartheta$-closed if and only if both $r^2=0$ and $s^2=0$.
\end{proof}

\subsection{Hermitian pluri-closed rank of \texorpdfstring{$6$}{6}-dimensional nilmanifolds}
Finally, we consider the pluri-closed rank $\mathrm{SKTr}(X)$ of $X$ as defined in \eqref{eq:SKT-rank} in Definition \ref{def:skt-rank}.

In case of solvmanifolds $X$ with associated Lie algebra $\mathfrak{g}$, a notion of {\em invariant pluri-closed rank} $\mathrm{SKTr}(\mathfrak{g})$ can be defined. Clearly, $\mathrm{SKTr}(\mathfrak{g})\leq \mathrm{SKTr}(X)$. Note that, in this case, the argument in the proof of Lemma \ref{lem:inv-ranks} does not apply. Indeed, there we make use of the map induced by the wedge product in the Morse-Novikov cohomology, $H^2_{d_\vartheta}(X)\times H^2_{d_\vartheta}(X) \to H^4_{d_{2\vartheta}}(X)$, which in turn is a consequence of the Leibniz rule for the twisted differential operator, namely $d_{k\vartheta}(\alpha\wedge\beta)=d_{h\vartheta}\alpha\wedge\beta+(-1)^{\mathrm{deg}\,\alpha}\,\alpha\wedge d_{(k-h)\vartheta}\beta$. But the $\partial\overline\partial$-operator, and the corresponding Aeppli cohomology, do not share these properties.

In the following Table \ref{table:ranks-nilmanifolds}, we show the invariant pluri-closed rank of $6$-dimensional nilmanifolds, summarizing also the ranks computed in the previous sections. The results follows by computing:

\begin{center}
  \begin{table}[ht]
  \centering
  \begin{tabular}{>{\bfseries\bgroup}l<{\bfseries\egroup} | >{$}c<{$} || >{$}c<{$} | >{$}c<{$} | >{$}c<{$} ||}
  \toprule
  \multicolumn{2}{c||}{\bfseries class} & \mathbf{\mathrm{Kr}(X)} & \mathbf{\mathrm{HlcKr}(X)} & \mathbf{\mathrm{SKTr}(\mathfrak{g})} \\
  \toprule
  \multirow{2}{*}{(P)} & \rho=0 & 3 & 3 & 3 \\
  & \rho=1 & 2 & 2 & 2 \\
  \midrule
  \multirow{2}{*}{(I)} & -\rho+D+\bar D-\lambda^2=0 & 2 & 2 & 3 \\
  & -\rho+D+\bar D-\lambda^2\neq0, \; (\rho,\lambda,D)\neq(0,0,-1) & 2 & 2 & 2 \\
  & (\rho,\lambda,D)=(0,0,-1) & 2 & 3 & 2 \\
  \midrule
  \multirow{2}{*}{(II)} & (\rho,B,c)\neq(1,1,0) & 1 & 2 & 2 \\
  & (\rho,B,c)=(1,1,0) & \geq1 & 2 & 2 \\
  \midrule
  (III) & & 1 & 1 & 1 \\
  \bottomrule
  \end{tabular}
  \caption{Special-Hermitian ranks for $6$-dimensional nilmanifolds endowed with invariant complex structures.}
  \label{table:ranks-nilmanifolds}
  \end{table}
  \end{center}

\begin{eqnarray*}
 \text{\bfseries(P):}&
     \quad\partial\overline\partial\omega \;=\;& -i t^2\rho\, \varphi^{12\bar1\bar2} \;; \\[5pt]
 \text{\bfseries(I):}&
     \quad\partial\overline\partial\omega \;=\;&
     \left(it^2(-\rho+D+\bar D-\lambda^2)\right) \,\varphi^{12\bar1\bar2} \;; \\[5pt]
 \text{\bfseries(II):}&
     \quad\partial\overline\partial\omega \;=\;&
     \left(-it^2(\rho+c^2+|B|^2)\right)\, \varphi^{12\bar1\bar2} \;; \\[5pt]
 \text{\bfseries(III):}&
     \quad\partial\overline\partial\omega \;=\;&
     (-2it^2)\, \varphi^{12\bar1\bar2}
     + (-2is^2)\, \varphi^{13\bar1\bar3} \;. \\[5pt]
\end{eqnarray*}

\begin{rmk}
From the results in Table \ref{table:ranks-nilmanifolds}, we note in particular the upper-semi-continuity of the Hermitian ranks.
We wonder whether this proerty holds in general.
\end{rmk}

\section{K\"ahler rank of a non-K\"ahler manifold constructed as suspension}
As another example, we consider here a non-K\"ahler manifold constructed as suspension over a torus.
We consider an explicit case of a more general construction which has been investigated by G.~\TH{}. Magnusson \cite{magnusson} to disprove the abundance and Iitaka conjectures for complex non-K\"ahler manifolds.
See also \cite[Example 3.1]{tosatti}, (and the references therein,) where V. Tosatti uses the same construction to get a complex non-K\"ahler manifold with vanishing first Bott-Chern class, whose canonical bundle is not holomorphically torsion.

\medskip

We first recall the construction by Yoshihara \cite[Example 4.1]{yoshihara} of a complex $2$-torus $X$ with an automorphism $f$ such that the induced automorphism on $H^0(X;K_X)\simeq\C$ has infinite order.

Consider the roots $\alpha\in\C$ and $\beta\in\C$ of the equation
$$ x^2-(1+\sqrt{-1})x+1 \;=\; 0 \;.$$
The minimal polynomial over $\Q$ of $\alpha$ and $\bar\beta$ is
$$ x^4-2x^3+4x^2-2x+1 \;=\; 0 \;.$$
In particular,
$$\left(\begin{array}{c}\alpha^4\\\bar\beta^4\end{array}\right)=2\left(\begin{array}{c}\alpha^3\\\bar\beta^3\end{array}\right)
-4\left(\begin{array}{c}\alpha^2\\\bar\beta^2\end{array}\right)+2\left(\begin{array}{c}\alpha\\\bar\beta\end{array}\right)-\left(\begin{array}{c}1\\1\end{array}\right)\;.$$

Consider the following lattice in $\C^2$:
$$ \Gamma \;:=\; \mathrm{span}_\Z \left\{
 \left(\begin{array}{c}1\\1\end{array}\right),\;
 \left(\begin{array}{c}\alpha\\\bar\beta\end{array}\right),\;
 \left(\begin{array}{c}\alpha^2\\\bar\beta^2\end{array}\right),\;
 \left(\begin{array}{c}\alpha^3\\\bar\beta^3\end{array}\right) \right\} \;.$$
Consider the torus
$$ X \;:=\; \left. \C^2 \middle\slash \Gamma \right. \;.$$

The automorphism
$$ f\colon \C^2\to \C^2 \;, \qquad f\left(\begin{array}{c}z_1\\z_2\end{array}\right) \;:=\; \left(\begin{array}{cc}\alpha&\\&\bar\beta\end{array}\right) \cdot \left(\begin{array}{c}z_1\\z_2\end{array}\right) $$
induces an automorphism of $X$.

Now, we recall the construction by G.~\TH{}. Magnusson \cite{magnusson} of the non-K\"ahler manifold $M$. 
Let
$$ C\;:=\;\left.\C \middle\slash (\Z\oplus\tau\Z)\right. $$
be an elliptic curve. Then $M$ is the total space of a holomorphic fibre bundle $M\to C$ with fibre $X$ as follows:
$$ M \;:=\; \left. X \times \C \middle \slash \Z^2\right. $$
where $\Z^2\circlearrowleft X\times\C$ acts as
$$ (\ell,m) \cdot (z,\,w) \;:=\; \left(f^m(z),\,w+\ell+m\,\tau\right)\;.$$
Note that $M$ is not K\"ahler, \cite[Proposition 1.2]{magnusson}, because of \cite[Corollary 4.10]{fujiki}.

\medskip

We claim that the K\"ahler rank of $M$ is equal to $1$.

In fact, note that the form
$$ d w \wedge d\bar w $$
on $\C$ yields a $d$-closed $(1,1)$-form of rank $1$ on $M$. Whence the K\"ahler rank of $M$ is greater than or equal to $1$.
On the other side, assume that there exists $\omega$ a $d$-closed $(1,1)$-form of rank at least $2$ on $M$.
It corresponds to a $d$-closed $\Z^2$-invariant $(1,1)$-form of rank at least $2$ on $X\times\C$. By the inclusion $\iota\colon X \ni x \mapsto (x,0) \in X\times \C$, it yields a $d$-closed $f$-invariant $(1,1)$-form of rank at least $1$ on $X$ --- say $\omega$ again.
Notice that $f$ sends invariant forms (with respect to the action of $\C^2$ on $X$) to invariant forms.
We have $\omega=\omega_{\text{inv}}+d\eta$ where $\omega_{\text{inv}}$ is invariant, and $\eta$ is a $1$-form. Then $f^*\omega=f^*\omega_{\text{inv}}+df^*\eta$, where $f^*\omega_{\text{inv}}$ is invariant. We get that $f^*\omega_{\text{inv}}=\omega_{\text{inv}}\mod d(\wedge^1\mathfrak{g}^\ast)$. Since the Lie algebra $\mathfrak{g}$ of $X$ is Abelian, we get $f^*\omega_{\text{inv}}=\omega_{\text{inv}}$.
We get that $\omega_{\text{inv}}$ is a $d$-closed invariant $f$-invariant $(1,1)$-form of rank at least $1$ on $X$. But this is not possible, since the only invariant $f$-invariant $(1,1)$-forms on $X$ are generated over $\C$ by $dz^1\wedge d\bar z^2$ and $dz^2\wedge d\bar z^1$.

\end{document}